\documentclass[12pt,oneside,reqno]{amsart}

\usepackage[a4paper,left=33mm,top=35mm,right=33mm,bottom=40mm]{geometry}

\usepackage{amssymb,amsfonts}
\usepackage[all,arc]{xy}
\usepackage{enumerate}
\usepackage{siunitx}
\usepackage{enumitem}
\usepackage{mathtools}
\usepackage{mathrsfs}
\usepackage{dsfont}
\usepackage{amsfonts}
\usepackage{amssymb}
\usepackage{graphicx}
\usepackage{ragged2e}
\usepackage{amsthm}
\usepackage{amsmath}
\usepackage[mathscr]{euscript}
 \let\mathscr\relax
\usepackage[scr]{rsfso}
\usepackage{graphicx}
\usepackage{color}
\usepackage{float}
\usepackage{caption}
\usepackage[pdftex,bookmarksnumbered,bookmarksopen]{hyperref}
\hypersetup{hidelinks = true}

\DeclareMathOperator{\vol}{Vol}
\DeclareMathOperator{\Var}{Var}

\newtheorem{thm}{Theorem}[section]
\newtheorem{cor}[thm]{Corollary}
\newtheorem{prop}[thm]{Proposition}
\newtheorem{lem}[thm]{Lemma}

\theoremstyle{definition}

\newtheorem{exmp}[thm]{Example}

\theoremstyle{remark}
\newtheorem{rem}[thm]{Remark}

\numberwithin{equation}{section}

\begin{document}
	\title[On mass distribution for toral eigenfunctions]{Mass distribution for Toral eigenfunctions via Bourgain's de-randomisation }
	\author{Andrea Sartori}
	\maketitle
	
	\begin{abstract}
We study the mass distribution of Laplacian eigenfunctions at Planck scale for the standard flat torus $\mathbb{T}^2=\mathbb{R}^2/\mathbb{Z}^2$. By averaging over the ball centre, we use Bourgain's de-randomisation to compare the mass distribution of toral eigenfunctions to the mass distribution of random waves in growing balls around the origin.  We then classify all possible limiting distributions and their variances. Moreover, we show that, even in the ``generic'' case, the mass might \textit{not} equidistribute at Planck scale.  Finally, we give  necessary and sufficient  conditions so that the mass of  ``generic" eigenfunctions equidistributes at Planck scale in almost all balls.  
	\end{abstract}

\section{Introduction}
\subsection{Shnirelman's Theorem in shrinking sets}
 Given a compact Riemannian manifold $(M,g)$ without boundary and normalised to have volume $1$, let $\Delta_g$ be the Laplace-Beltrami operator on M. Then, there exists an orthonormal basis for $L^2(M,dvol)$ consisting of eigenfunctions $\{f_{E_i}\}$
 
 \begin{align}
 	\Delta_g f_{E_i}+ E_i f_{E_i}=0  \nonumber
 \end{align}
 with $0=E_1<E_2\leq...$ repeated accordingly to multiplicity, and $E_i\rightarrow \infty$. The celebrated Quantum Ergodicity Theorem \cite{DV,S,Z} asserts that, if the geodesic flow on M is ergodic, then there exists a density one subsequence of eigenfunctions $\{f_{E_j}\}$ of $\Delta_g$  such that 

\begin{align}
\int_{A}|f_{E_j}|^2 \underset{j\rightarrow \infty}{\longrightarrow} Vol(A) \label{QE}
\end{align}
where $A$ is an open subset of $M$. That is, the $L^2$ mass of most eigenfunctions equidistributes on $M$. Berry random waves model \cite{B1,B2} implies a stronger form of this fact: given a parameter $r=r(E)$ such that $r \cdot\sqrt{E}\rightarrow \infty$, we expect, for \textit{generic} eigenfunctions, that 
\begin{align}
	\frac{1}{Vol (B(x,r))}\int_{B(x,r)}|f_{E}|^2\underset{E\rightarrow \infty}{\longrightarrow} 1 \label{SQE}
\end{align}
uniformly for all $x\in M$ (here $ B(x,r)$ denotes the geodesic ball centred at $x$ of radius $r$). Thus, the mass of \textit{generic} eigenfunctions should equidistribute at \textit{Planck scale}.  

We are interested in the case when $M$ is the flat torus $\mathbb{T}^2=\mathbb{R}^2/\mathbb{Z}^2$. Although the geodesic flow on $\mathbb{T}^2$ is completely integrable,  Lester and Rudnick \cite{LS} proved that (\ref{SQE}) holds for a density one subsequence of eigenfunctions at scale $r> E^{-1/2+ o(1)}$. On the other hand, they also proved that there exist eigenfunctions for which  (\ref{SQE}) fails at the point $x=0$ at Planck scale. This naturally raises the question of whether the  failure of  (\ref{SQE}) is only limited to a small set of ball centres. To make this precise, Granville and Wigman \cite{GW} introduced the (pseudo-)random variable

\begin{align}
M_f(x,r):=	\frac{1}{Vol (B(x,r))}\int_{B(x,r)}|f_{E}|^2 \label{defM}
\end{align}
where $x$ is drawn uniformly at random from $\mathbb{T}^2$ and showed, under some additional assumptions on $f$ (see Remark \ref{comparison2} below), that the variance of $M_f(x,r)$ tends to zero at Planck scale. Therefore, (\ref{SQE}) holds for most points $x\in \mathbb{T}^2$. Furthermore, Wigman and Yesha \cite{WY} proved, under flatness assumptions and small variation of the coefficients (see Remark \ref{comparison2} below), that the distribution of $M_f(x,r)$ is asymptotically Gaussian with mean zero and variance $ c \cdot  (\sqrt{E}r)^{-1}$, where the constant $c$ depends on the eigenfunction $f$.  

Bourgain \cite{B1} observed that ``generic" toral eigenfunctions, when averaged over $\mathbb{T}^2$, are comparable to a Gaussian random field. We apply the so called Bourgain's de-randomisation to study $M_f(x,r)$. This allows us to find  its limit distribution and variance for a wider class of eigenfunctions than \cite{GW,WY}. Via the study of the variance, we also show that, even for ``generic" sequences of eigenfunctions, the mass might \textit{not} equidistribute at Planck scale. Moreover, we are able to give sufficient and necessary conditions for mass equidistribution which include and extend some of the results from \cite{GW,WY}.

 \vspace{3mm}\paragraph{\textit{Related results}} When $M$ is the modular surface, Luo and Sarnak \cite{LS} proved that there exists a density one subsequence of eigenfunctions, which are also eigenfunctions of all Hecke operators, such that small scale equidistribution holds for $r\gg E^{-\alpha}$ for some small $\alpha>0$. Young \cite{Y} improved, under the Generalised Riemann Hypothesis, the scale of the said result to $r\gg E^{-1/6+o(1)}$ and to \textit{all} eigenfunctions. Hezari, Rivi\`{e}re \cite{HR}  and independently Han \cite{H} proved that, if $M$ has negative sectional curvature, then small scale equidistribution holds along a density one subsequence of eigenfunctions for $r= \log(E)^{-\alpha}$ for some small $\alpha>0$. Finally, such an approach of averaging over ball centres was also recently used by Humphries \cite{HU} for mass distribution of automorphic forms.

\subsection{Toral eigenfunctions}
\label{random waves}
An eigenfunction for $\Delta$ on $\mathbb{T}^2$ with eigenvalue $E$  can be written as

\begin{align}
	f(x)=\sum_{ |\xi|^2=E}a_{\xi}e(\langle x,\xi \rangle)  \label{function}
\end{align}
where \footnote{ This normalisation implies that the eigenvalue is $4\pi^2 E$.} $e(\cdot)= e(2\pi i \cdot )$   and some complex numbers $\{a_{\xi}\}_{\xi}$.  The multiplicity of the eigenvalue is  the number of representations of $E$ as a sum of two squares and we denote it by $N=N(E)$, we also let $\mathcal{E}=\mathcal{E}(E)=\{\xi \in \mathbb{Z}^2: |\xi|^2=E\}$ so that $|\mathcal{E}|=N$. To assure that $f$ is real-valued, we assume

\begin{align} 
 \overline{a_{\xi}}=a_{-\xi}. \label{simmetry} 
 \end{align}
 Moreover, we normalise $f$ so that

\begin{align}
||f||^2_{L^2(\mathbb{T}^2)}=\sum_{ |\xi|^2=E}|a_{\xi}|^2=1 \label{normalisation} . 
\end{align}
Thanks to (\ref{normalisation}), we can associate to $f$ a probability measure on the unit circle $\mathbb{S}^1= \mathbb{R}/\mathbb{Z}$

\begin{align}
	\mu_{f}=\sum_{ |\xi|^2=E} |a_{\xi}|^2\delta_{\xi/\sqrt{E}} \label{spectral measure}
\end{align} 
  where $\delta_{\xi/\sqrt{E}}$ is the Dirac delta function at the point $\xi/\sqrt{E}$.  The measures $\mu_{f}$ appear naturally in the study of $f$:  Bourgain \cite{BU} and subsequently Buckley and Wigman \cite{BW} proved that ``generic'' eigenfunctions $f$,  when averaged over $x\in \mathbb{T}^2$, approximate a centred stationary Gaussian field with spectral measure  $\mu_{f}$. 
 
 \begin{rem}
 Importantly, the sequence  $\mu_{f}$ does not have a unique weak$^\star$ limit. In fact, the weak limits of $\{\mu_f\}$, in the special case  $|a_{\xi}|^2=1/N$ for all $\xi$, are called ``attainable'' and have been studied in \cite{KW}. 
 \end{rem}
 \subsection{Gaussian fields}
 \label{Gaussian fields}
We  briefly collect some facts about Gaussian fields (on $\mathbb{R}^2$) which will be used later.  A (real-valued) Gaussian field $F$ is a measurable map $F:\Omega \times \mathbb{R}^2\rightarrow \mathbb{R}$ for some probability space  $\Omega$,  such that all finite dimensional distributions of $(F(x_1, \cdot),...F(x_n,\cdot))$ are multivariate Gaussian, where $x_i\in \mathbb{R}^2$ and $n\in \mathbb{N}$. Moreover, $F$ is \textit{centred} if $\mathbb{E}[F]=0$ and \textit{stationary} if its law is invariant under the action $x\rightarrow x+\tau$ for $\tau \in \mathbb{R}^2$.

 By Kolmogorov theorem, every centred Gaussian field is fully determined by its covariance function  

\begin{align}
\mathbb{E}[F(x)\cdot F(y)]= \mathbb{E}[F(x-y)\cdot F(0)] \nonumber
\end{align}
and the stationary property is equivalent to the fact that the covariance function only depends on the difference of $x$ and $y$. Furthermore, the covariance is positive definite so, by Bochner's theorem, it is the Fourier transform of some measure $\mu$ on the plane satisfying $\mu(I)=\mu(-I)$ (as the field is real-valued); that is

\begin{align}
\mathbb{E}[F(x)F(y)]= \int e\left(\langle x-y, \lambda \rangle\right)d\mu(\lambda). \nonumber
\end{align}
The measure $\mu$ is called the \textit{spectral measure} of $F$. Since Gaussian fields are determined by their mean and covariance, $\mu$ fully determines $F$ when $F$ is centred, and we may write $F=F_{\mu}$. From this point on, it will be tacitly assumed that  all random fields are defined on a common probability space $(\Omega,\mathcal{F},\mathbb{P})$ with expectation $\mathbb{E}$.  
 
 \subsection{Statement of main results}
 \label{Statement}
We make the following two assumptions, which we will discuss in Section \ref{passage to random fields}:
\begin{itemize}
	\setlength\itemsep{2mm}
	\item A1 (Spectral correlations). Let $0<\gamma< 1/2$, $E$ be an integer representable as the sum of two squares and   $B=B(E)$ be an arbitrarily slow growing function of $E$ taking integer values.  Then, we say that $E$ satisfies assumption $A1$ if for every  $2\leq 2l\leq B$ the number of $2l$-tuples $(\xi_1,...,\xi_{2l})\in \mathcal{E}(E)^{2l}$ satisfying
	 \begin{align}
	 	\xi_1+ \xi_2+...+\xi_{2l}=0 \label{spectral}
	 \end{align}
	 is 
	 \begin{align}
	 	\frac{(2l!)}{2^l \cdot l!}N^l + O(N^{\gamma l}). \nonumber
	 \end{align}
	 where the constant implied in the notation is absolute. 

	 \item A2 (Flatness). Fix some function $u: \mathbb{R}\rightarrow \mathbb{R}$ such that for any $\epsilon>0$ $u(N)/N^{\epsilon} \rightarrow 0$ as $N\rightarrow \infty$. A function $f$ of the form (\ref{function}), normalised as in (\ref{normalisation}),  satisfies assumption A2 if for any $\xi\in \mathcal{E}$ 
	 
	 \begin{align}
	 	|a_{\xi}|^2 \leq \frac{u(N)}{N}. \label{4}
	 \end{align}
\end{itemize}   
\begin{rem}
	By \cite[Theorem 17]{BB} (see also \cite[Lemma 4]{BU}) assumption A1 is satisfied for a density one subsequence of energy levels.  Moreover, assumption A1, for \textit{every} eigenvalue $E$, would follow from a sub-exponential bound in the deep work of  Evertse-Schlickewei-Schmidt on additive relations in multiplicative subgroups of $\mathbb{C}^{\star}$ of bounded rank \cite[Remark 1]{BU}. If we regard the set $(a_{\xi})_{\xi}$ as points on an $N$-dimensional (complex) sphere, by L\'{e}vy concentration of measure \cite[Theorem 2.3]{LE}, assumption A2 (with $u(N)=(\log N)^{O(1)}$, say) is satisfied with probability asymptotic to $1$. Thus, ``generic''  eigenfunctions on $\mathbb{T}^2$ satisfy both conditions.
\end{rem} 
We will then prove the following: 
\begin{thm}
	\label{main theorem}
	Let $\epsilon>0$ and $t>0$, $R>1$ be fixed, $M_f(x,r)$ be as in (\ref{defM}), $\mu_{f}$ as in (\ref{spectral measure}) and $F_{\mu_{f}}$ as in Section \ref{Gaussian fields}. Then there exists some $E_0=E_0(\epsilon,t,R)$ such that for every $E> E_0$ satisfying A1 and $N\rightarrow \infty$ we have 
	
	\begin{align}
		\left| \vol (x\in \mathbb{T}^2: M_f(x,r)\leq t)- \mathbb{P}\left( \frac{1}{\pi R^2}\int_{B(R)}|F_{\mu_{f}}|^2\leq t\right)\right|\leq \epsilon \nonumber
	\end{align}
	where $r=R/\sqrt{E}$,  uniformly for all $f$ satisfying A2. 
\end{thm}

Upon  passing to a subsequence,  we can assume that $\mu_{f}$ weak$^{\star}$ converges to  $\mu$ ($\mu_{f}\Rightarrow \mu$), where $\mu$ is some probability measure on $\mathbb{S}^1$. We are then able to find the distribution of $(1/\pi R^2)\int_{B(R)}|F_{\mu}|^2$, see Section \ref{distribution randon waves}. This leads to the next result which requires some extra notation.  By Lebesgue decomposition theorem \cite[Theorem 19.61]{HR}, we can write  

\begin{align}
\mu =\alpha\mu_{A} + \beta \mu_B  \label{decomp}
\end{align} 
where $\alpha,\beta \in \mathbb{R}$, $\mu_A$ is a purely atomic probability measure and $\mu_{B}$ is a probability measure with no atoms. We then write $\mu_A = \sum_{\xi \in \text{spt}(\mu_A)} \sigma_{\xi}\delta_{\xi}$ 
for some ${\sigma_{\xi}}>0$ with $\sum_{\xi} \sigma_{\xi}=1$ and define

 \begin{align}
 W(\mu):= \sum_{ \xi \in \text{spt}(\mu_A)} |X_{\xi}|^2 \label{W}
 \end{align} where $ X_{\xi}$ are i.i.d. $N(0,\sigma_{\xi})$ complex random variables satisfying $\overline{X_{\xi}}= X_{-\xi}$. 

\begin{thm}
	\label{main corollary}
Let $\mu$ be a probability measure on $\mathbb{S}^1$, write $\mu= \alpha \mu_{A} + \beta \mu_B$ as in (\ref{decomp}) and let $t>0$. Then, uniform for all $f$ satisfying $A2$ such that $\mu_{f} \Rightarrow \mu$, we have  
	\begin{align}
\lim\limits_{R\rightarrow \infty}\lim\limits_{E \rightarrow \infty} \vol(x\in \mathbb{T}^2: M_f(x,R/\sqrt{E})\leq t) = \mathbb{P}\left( \alpha \cdot W(\mu_A) + \beta \leq t \right) \nonumber
\end{align}
where the limit $E\rightarrow \infty$ is taken over a sequence of eigenvalues satisfying $A1$. Moreover,
\begin{align}
\lim\limits_{R\rightarrow \infty}\lim\limits_{E \rightarrow \infty}\int_{\mathbb{T}^2}( M_f(x,R/\sqrt{E})-1)^2 = \alpha^2 \Var(W(\mu_A)). \nonumber
\end{align}
\end{thm}

 Taking $\alpha=0$ and thus $\beta=1$ in Theorem \ref{main corollary}, we have: 
\begin{cor}
	\label{noatoms}
	Under the assumptions of Theorem \ref{main corollary}, as $E\rightarrow \infty$ and  $R\rightarrow \infty$, we have 
	
	\begin{align}
	 	\int_{\mathbb{T}^2}( M_f(x,r)-1)^2 \rightarrow 0 \nonumber
	\end{align}
if and only if $\mu$ has no atoms. 
\end{cor}

\begin{rem}
	\label{comparison2}
	In \cite{GW,WY} the authors give sufficient conditions for the vanishing of the variance of $M_f(x,R/\sqrt{E})$. These are essentially flatness of $f$ and the lack of mass concentration for $\mu_{f}$ (see  \cite[equation (19)]{GW} and \cite[Definition $2.4$]{WY}). In hindsight, Corollary \ref{noatoms} explains  such conditions and shows that they are sharp: they imply that $\mu$ is non-atomic. 
\end{rem}

\begin{exmp}
	\label{atoms}
	Let us consider ``Bourgain's eigenfunctions": 
	
	\begin{align}
	f(x)=	\frac{1}{\sqrt{N}}\sum_{ |\xi|^2=E} e(\langle \xi, x\rangle). \nonumber
	\end{align}
 Such eigenfunctions satisfy $A2$, so for the rest of this example we assume that the eigenvalues satisfy $A1$ and $N\rightarrow \infty$.  Cilleruelo \cite{C} proved that there exists a sequence of eigenvalues $E$ such that $N\rightarrow \infty$  and $\mu_f\Rightarrow\mu:=\frac{1}{4}(\delta_{(1,0)} + $ $ \delta_{(0,1)}+ \delta_{(-1,0)}  + \delta_{(0,-1)})$. Moreover, using the arguments in \cite{BMW}, one can show that $\mu$ can be attained even after restricting to eigenvalues satisfying $A1$. Thus, for Bourgain's eigenfunctions such that $\mu_{f} \Rightarrow \mu$, Theorem \ref{main corollary} asserts that

\begin{align}
	& M_f(x,r) \overset{d}{\longrightarrow}  \frac{\chi^2(2)}{2} &\int_{\mathbb{T}^2}( M_f(x,r)-1)^2 \rightarrow 1 \nonumber
\end{align}
as  $E\rightarrow \infty$ and $R\rightarrow \infty$, where $\chi^2(2)$ is the Chi-squared distribution with $2$ degrees of freedom and the convergence is in distribution. Hence, mass equidistribution does not hold in this case. 
\end{exmp}

\subsection{Notation}
Let $x \rightarrow \infty$ be some parameter, we say that the quantity $X=X(x)$ and $Y=Y(x)$   satisfy $X\ll Y$ , $X\gg Y$ if there exists some constant $C$, independent of $x$, such that $X\leq C Y$ and $X\geq CY$ respectively. We also write $O(X)$ for some quantity bounded in absolute value by a constant times $X$ and $X=o(Y)$ if $X/Y\rightarrow 0$ as $x\rightarrow \infty$, in particular we denote by $o(1)$ any function that tends to $0$ as $x\rightarrow \infty$. We denote by $\Rightarrow$ the weak$^\star$ convergence of probability measures, by $\overset{d}{\longrightarrow}$ convergence in distribution and by $B(R)$ the two dimensional ball centred at $0$ of radius $R>0$. Moreover, $\mu$ will always denote a probability measure supported on $\mathbb{S}^1$ (the unit circle) and in general $\mu_{f} \Rightarrow \mu$. Finally, we recall that $\Omega$ is the abstract probability space where all random objects are defined.

\section{Proof of Theorem \ref{main corollary}}
In this section  we prove Theorem \ref{main corollary}, assuming Theorem \ref{main theorem}. This will be done in two steps: first we take the limit as $E\rightarrow \infty$ and show that the limiting distribution of $M_f(x,R/\sqrt{E})$ tends to the random variable  $\frac{1}{\pi R^2}\int_{B(R)}|F_{\mu}|^2$. Secondly, we determine such random variable in the limit $R\rightarrow \infty$. 
\subsection{Limit as $E\rightarrow \infty$}
The aim of this section is to prove the following: 
\begin{prop}
	\label{conv dis}
	Under the assumptions of Theorem \ref{main corollary}, and $\mu_{f}\Rightarrow \mu$. Then, as $E\rightarrow \infty$, we have 
	\begin{align}
	M_f(x,r) \overset{d}{\longrightarrow}\frac{1}{\pi R^2}\int_{B(R)}|F_{\mu}|^2 dx. \nonumber
	\end{align}
\end{prop}

 To prove Proposition \ref{conv dis}, we need two results. The first one is a direct consequence of the Borell-TIS inequality  \cite[Theorem 2.1.1]{AT}, which states that if a Gaussian  field, $F$ is almost surely bounded on $B(R)$ then $\mathbb{E}[\sup_{B(R)}|F|]<\infty$.
\begin{lem}
	\label{not too big} 
	Let $\mu$ be a probability measure on $\mathbb{S}^1$. For any $\delta_1>0$ there exists some $M=M(R,\delta_1)$ such that 
	
	\begin{align}
	\mathbb{P}\left(	\underset{B(R)}{\sup}|F_{\mu}|>M\right)\leq \delta_1 \nonumber.
	\end{align}
\end{lem}
\begin{proof}
	Since $\mu$ is supported on $\mathbb{S}^1$, we have $\int |\lambda|^3 d\mu(\lambda)\leq 1$, thus the covariance function is differentiable; then $F_{\mu}$ is almost surely continuous and thus almost surely bounded in $B(R)$. The
	Borell-TIS inequality \cite[Theorem 2.1.1]{AT} then implies that $\mathbb{E}[\sup_{B(R)}|F_{\mu}|]$ is finite, therefore we can take $M= \delta^{-1}_1  \mathbb{E}[\sup_{B(R)}|F_{\mu}|]$ and apply Markov's inequality. 
\end{proof}

We also need the following lemma from \cite{SO} of which we give the proof for convenience.
\begin{lem}[Lemma 4, \cite{SO}]
	\label{Sodin}
	Let $\{\mu_n\}_{n\in \mathbb{N}}$ be a sequence of probability measures on $\mathbb{S}^1$ such that $\mu_n \Rightarrow \mu$. Then, for any $\alpha>0$ and $\delta_2>0$,  there exists some $n_0=n_0(\alpha,\delta_2)$ such that for all $n\geq n_0$ we have 
	
	\begin{align}
	\underset{B(R)}{\sup}|F_{\mu_n}- F_{\mu}|\leq \alpha \nonumber
	\end{align}
	outside an event of probability $\delta_2$. 
\end{lem}

	\begin{proof}
	We can associate to $\mu$ the random field $G$ defined on $\mathbb{R}^2$ as follows: for any open and measurable (with respect to $\mu$) subset $A$ of $\mathbb{R}^2$ we let
	
	\begin{align}
	G(A)= N(0,\mu(A)).\nonumber
	\end{align}
Moreover, if $A\cap B= \emptyset$, we require $G(A)$ and $G(B)$ to be independent. We define $G_n$ with respect to $\mu_n$ similarly. Since $\mu$ is compactly supported, we see that $G_n \overset{d}{\rightarrow}G$ and, since a normal random variable is square integrable, we obtain $G_n\rightarrow G$ in $L^2(\Omega)$ (we recall that $\Omega$ is the common probability space of our random objects). By \cite[Theorem 5.4.2]{AT}, we have the $L^2(\Omega)$ representations 
	\begin{align}
	&F_{\mu_n}(x)= \int_{\mathbb{S}^1}e(\langle x\cdot\lambda\rangle)G_n(d\lambda) &F_{\mu}(x)= \int_{\mathbb{S}^1}e(\langle x\cdot\lambda\rangle)G(d\lambda). 	\nonumber
	\end{align}
	From this and the preceding discussion, we deduce that	$\sup_{B(R)}|F_{\mu_n}- F_{\mu}|\rightarrow 0$ in $L^2(\Omega)$ from which the lemma follows. 
\end{proof}

\begin{proof}[Proof of Proposition \ref{conv dis}]

		Fix some $t\in \mathbb{R}$, $\epsilon>0$ and let $X(\mu_{f})= \int_{B(R)}|F_{\mu_{f}}|^2/\pi R^2$. Then, by L\'{e}vy continuity Theorem \cite[Theorem 4.3]{K}, the proposition is equivalent to 
		
			\begin{align}
		\left|\int_{\mathbb{T}^2} \exp ( it M_f(x,r))- \mathbb{E}[ \exp(it  X(\mu))]\right|\leq \epsilon \nonumber
		\end{align}
		for all $E$ large enough.  By Theorem \ref{main theorem} and L\'{e}vy continuity Theorem, we have
		
		\begin{align}
		\left|\int_{\mathbb{T}^2} \exp ( it M_f(x,r))- \mathbb{E}[ \exp(it  X(\mu_{f}))]\right|\leq \epsilon/2 \nonumber
		\end{align}
		for all $E$ large enough. Thus, by the triangle inequality, it suffices to prove 
		
		\begin{align}
		\left|\mathbb{E}[ \exp(it  X(\mu_{f}))]- \mathbb{E}[ \exp(it  X(\mu))]\right|\leq \epsilon/2 \label{3}
		\end{align} 
		for all $E$ large enough. We are going to show that $ X(\mu_{f})$ and $X(\mu)$ are close outside and event of small probability. To see this, we take $\delta_1=\epsilon/8$ in Lemma \ref{not too big} and let $V_1$ be the event that $\sup|F_{\mu}|>M$. We take  $\delta_2=\epsilon/8$, $\mu_n=\mu_{f}$, $\alpha=\sqrt{\epsilon}/(8tM)$  in Lemma \ref{Sodin} and denote by $V_2$ the event that $\sup|F_{\mu_f}- F_{\mu}|> \alpha$.Let $V= V_1\cup V_2$ then $\mathbb{P}(V)\leq \epsilon/4$ for all $E$ large enough. Moreover, outside $V$, we have 
		
		\begin{align}
		&|X(\mu_{f}) -X(\mu)| = \frac{1}{\pi R^2}\int_{B(R)} |F_{\mu}- (F_{\mu}- F_{\mu_{f}})|^2- |F_{\mu}|^2 \nonumber \\
		& \leq  	2\underset{B(R)}{\sup}|F_{\mu_f}- F_{\mu}| |F_\mu|+ 		\underset{B(R)}{\sup}|F_{\mu_f}- F_{\mu}|^2\leq 2M \alpha + \alpha^2\leq \epsilon/4t \label{5}
		\end{align}
and
		\begin{align}
		\left| \exp(it  X(\mu_{f}))- \exp(it  X(\mu))\right|\leq 	\left| \exp(it  X(\mu_{f})- X(\mu))- 1\right| \leq t | X(\mu_{f})- X(\mu)|. \label{6}
		\end{align}
		Combining (\ref{5}) and (\ref{6}) with the fact that $\mathbb{P}(V)\leq \epsilon/4$ we obtain (\ref{3}) as required.
\end{proof}
\subsection{Mass distribution for random waves}
\label{distribution randon waves}
In this section we  study the distribution of the random variable $ \int_{B(R)} |F_\mu|^2/\pi R^2$ as $R\rightarrow \infty$, where $\mu$ is a probability measure on $\mathbb{S}^1$. 
\begin{prop}
	\label{massdistribution}
Suppose that $\mu$ is a probability measure on $\mathbb{S}^1$, write $\mu $ as in (\ref{decomp}) and let $W(\mu_A)$ be as in (\ref{W}).  Then,  as $R\rightarrow \infty$, we have 
 	\begin{align}
 \frac{1}{\pi R^2}\int_{B(R)}|F_{\mu}|^2 \overset{d}{\longrightarrow} \alpha W(\mu_A) +\beta . \nonumber
 \end{align}
\end{prop}
To prove Proposition \ref{massdistribution}, we need a few preliminary results. 
 \begin{lem}
	\label{decomposing distribution}
	Suppose that $\mu$ is a probability measure on $\mathbb{S}^1$ and write $\mu $ as in (\ref{decomp}).  Then 
	
	\begin{align}
F_{\mu}= \sqrt{\alpha}F_{\mu_{A}} + \sqrt{\beta}F_{\mu_B} \label{21}
\end{align}
where $F_{\mu_{B}}$ and $F_{\mu_A}$ are independent, centred  stationary Gaussian fields.  
\end{lem}
\begin{proof}
Since both the left and the right hand side of (\ref{21}) are Gaussian fields with mean zero, they are fully determined by their covariances.  The covariance of the right hand side is 

\begin{align}
&\mathbb{E}[\left(\sqrt{\alpha}F_{\mu_{A}} + \sqrt{\beta}F_{\mu_B}\right)(x)\left(\sqrt{\alpha}F_{\mu_{A}} + \sqrt{\beta}F_{\mu_B}\right)(y)]  \nonumber \\ 
&= \alpha \mathbb{E}[F_{\mu_{A}}(x)F_{\mu_B}(y)] + \beta\mathbb{E}[F_{\mu_{A}}(x)F_{\mu_B}(y)] \nonumber \\ 
&= \int_{\mathbb{S}^1} e\left(\langle x-y, \lambda \rangle\right) d(\alpha\mu_{A}(\lambda)+ \beta\mu_B(\lambda)) \nonumber
\end{align}
where in the second line we have used the fact that the cross term vanishes by independence and in the last line we have used the definition of $F_{\mu_{B}}$ and $F_{\mu_A}$. Therefore, the spectral measure of $\sqrt{\alpha}F_{\mu_{A}} + \sqrt{\beta}F_{\mu_B}$ is equal to the spectral measure of $F_{\mu}$ and this proves the Lemma. 
\end{proof}
We can understand the distribution of the atomic part of $F_{\mu}$ directly as follows:  
\begin{lem}
	\label{atomic}
	Suppose that $\mu_A$ is a purely atomic measure supported on $\mathbb{S}^1$. Then  
	\begin{align}
\frac{1}{\pi R^2}\int_{B(R)}|F_{\mu_A}|^2 \overset{d}{\longrightarrow} W(\mu_A) \nonumber
\end{align}
as $R \rightarrow \infty$ where $W(\mu_A)$ is as in (\ref{W}).
\end{lem}

\begin{proof}
We claim that we can write explicitly 
\begin{align}
F_{\mu_A}(x)= \sum_{\xi \in \text{spt}(\mu_A)}X_{\xi}e(\langle\xi , x\rangle) \label{22}
\end{align}	
where $X_{\xi}$ are as in (\ref{W}). First, by Kolmogorov's two-series theorem \cite[Lemma 3.16]{K},  the sum in (\ref{22}) absolutely converges almost surely because the sum of the variances of the $X_{\xi}$'s is $\sum_{ \xi} \sigma_{\xi}=1$. Thus, we can compute the covariance function $r(x-y)= \sum \sigma_{\xi}e(\langle x-y , \xi\rangle)$, and observe that it is the Fourier transform of $\mu$. This proves the claim. 

We square $F_{\mu_A}(x)$ in (\ref{22}) and separate the diagonal terms to obtain 
\begin{align}
\frac{1}{\pi R^2}\int_{B(R)}|F_{\mu_A}|^2 & = \sum_{ \xi}|X_{\xi}|^2 + \frac{1}{\pi R^2}\sum_{\xi\neq \xi'}X_{\xi}\overline{X_{\xi'}}\int_{B(R)}e(\langle\xi-\xi',x\rangle)dx \nonumber \\
&= W(\mu_A)+ O\left(\sum_{\xi\neq \xi'}X_{\xi}\overline{X_{\xi'}}\int_{B(1)} e(\langle\xi-\xi', Rx\rangle) dx \right) \nonumber \\
&= W(\mu_A) + O\left(\sum_{\xi\neq \xi'}X_{\xi}\overline{X_{\xi'}}\frac{J_1(R||\xi-\xi'||)}{R||\xi-\xi'||}\right) 
\label{8}
\end{align}
where $J_1(\cdot)$ is the Bessel function of the first kind. We observe\footnote{In general it is not true that for any two series of random variables $X_n\overset{d}{\longrightarrow} X$ and $Y_n\overset{d}{\longrightarrow}Y$ then $X_n+Y_n\overset{d}{\longrightarrow} X+Y$.} that, in order to prove the lemma, it is enough to prove that the  error term in (\ref{8}) converges in distribution to 0. Indeed, writing $Z$ for the  error term in (\ref{8}), if $Z\overset{d}{\longrightarrow} 0$ then $Z$ converges in probability to $0$. Thus, the vector $(W(\mu_{A}),Z)$ converges in distribution to the vector $(W(\mu_{A}),0)$. Therefore, by the continuous mapping theorem, $W(\mu_{A})+Z \overset{d}{\longrightarrow} W(\mu_{A})$. 

 Now, we are going to prove that $Z\overset{d}{\longrightarrow}0$. Since $J_1(T)\ll T^{-1/2}$ for $T$ large and $J_1(T)/T=O(1)$ for $T$ small, $Z$ can be bounded as 
\begin{align}
	\sum_{\xi\neq \xi'}X_{\xi}\overline{X_{\xi'}}\frac{J_1(R||\xi-\xi'||)}{R||\xi-\xi'||} \ll R^{-3/4}	\sum_{|\xi- \xi'|>R^{-1/2}}|X_{\xi}||\overline{X_{\xi'}}| +\sum_{\substack{|\xi- \xi'|<R^{-1/2} \\ \xi\neq \xi'}}|X_{\xi}||\overline{X_{\xi'}}|. \label{17}
\end{align}
Since $\mathbb{E}[ \sum_{\xi\neq \xi'} |X_{\xi}\overline{X_{\xi'}}|] \leq (2/\pi)^{1/2} \sum_{\xi\neq \xi'} \sigma_{\xi}\sigma_{\xi'}\leq 1 $, we can apply Markov inequality to obtain 

\begin{align}
\mathbb{P}\left(\sum_{\xi\neq \xi'} |X_{\xi}||\overline{X_{\xi'}}|>R^{1/2}\right)\leq 1/R^{1/2}. \nonumber
\end{align}
Therefore, the first term on the right hand side of (\ref{17}) tends to zero outside an event of probability $1/R^{1/2}$. To bound the second term in (\ref{17}), we observe
\begin{align}
\mathbb{E}\left[ \sum_{\substack{|\xi- \xi'|<R^{-1/2} \\ \xi\neq \xi'}} |X_{\xi}\overline{X_{\xi'}}|\right]\ll \sum_{ \xi}\sigma_{\xi} \sum_{\substack{|\xi- \xi'|<R^{-1/2} \\ \xi\neq \xi'}}\sigma_{\xi'}\longrightarrow 0 \nonumber
\end{align}
as $R\rightarrow \infty$, because (for fixed $\xi$)  $\sum\sigma_{\xi'}=o(1)$ and $\sum\sigma_{\xi}\leq 1$. It follows that also the second term in (\ref{17}) converges in distribution to $0$. Hence, the error term in \eqref{8} converges in distribution to zero and the lemma follows. 
\end{proof}
Understanding the non-atomic part of $F_{\mu}$ requires  the following deep theorem due to Wiener and  Grenander-Fomin-Maruyama (see \cite[Theorem 3]{SO} and references therein): 

\begin{thm}
	\label{ergodic}
	Suppose that the spectral measure $\mu$ (supported on $\mathbb{S}^1$) of a  stationary Gaussian field $F=F_{\mu}$  has no atoms. Then, for any random variable $H(F)$ with $\mathbb{E}[|H(F)|]<\infty$, we have 
	\begin{align}
	\lim\limits_{R\rightarrow \infty} \frac{1}{\pi R^2}\int_{B(R)}H(F(x))dx= \mathbb{E}[H(F)]. \nonumber
	\end{align}
	almost surely and in $L^1$. 
\end{thm}
Taking $H(F)= F(0)^2$ so that $H(F(x))= H(F(0+x))= F(x)^2$ and  $\mathbb{E}[|F(0)^2|]= 1$,  we obtain the following lemma:
 
 \begin{lem}
 	\label{noatomic}
 		Suppose that $\mu_{B}$ is a probability measure supported on $\mathbb{S}^1$ with no atoms. Then, 	as $R \rightarrow \infty$, we have
 	\begin{align}
 	\frac{1}{ \pi R^2}\int_{B(R)}|F_{\mu_{B}}|^2 \overset{d}{\longrightarrow} 1 .\nonumber
 	\end{align}
 \end{lem}
We are now ready to give the proof of Proposition \ref{massdistribution}: 
\begin{proof}[Proof of Proposition \ref{massdistribution}]
 By Lemma \ref{decomposing distribution}, we can write
 \begin{align}
 \frac{1}{\pi R^2}\int_{B(R)}|F_{\mu}|^2&= \alpha 	\frac{1}{\pi R^2}\int_{B(R)}|F_{\mu_{A}}|^2 + \beta 	\frac{1}{\pi R^2}\int_{B(R)}|F_{\mu_B}|^2 \nonumber \\
 & +2\sqrt{\alpha\beta}	\frac{1}{\pi R^2}\int_{B(R)}(F_{\mu_{A}}\cdot F_{\mu_B}). \label{16}
 \end{align}
By Lemma \ref{atomic}, the first term in (\ref{16}) converges in distribution to $W(\mu_A)$ as $R\rightarrow \infty$. By Lemma \ref{noatomic}, the second term converges in distribution to $\beta$. Therefore, arguing as in Lemma \ref{atomic},  it is enough to prove
  
  \begin{align}
  \frac{1}{\pi R^2}\int_{B(R)}(F_{\mu_{A}}\cdot F_{\mu_B})dx \overset{d}{\longrightarrow}0. \label{33}
  \end{align}
 By the independence of $F_{\mu_{B}}$ and $F_{\mu_A}$,  we have
 \begin{align}
 \mathbb{E}\left[\frac{1}{\pi R^2}\int_{B(R)}(F_{\mu_A}\cdot F_{\mu_B})dx\right]= 0 \label{34}
 \end{align}
 and 
 \begin{align}
 &\mathbb{E}\left[\left(\frac{1}{\pi R^2}\int_{B(R)}(F_{\mu_A}\cdot F_{\mu_B})dx\right)^2\right]= \nonumber \\
 =& \frac{1}{\pi^2 R^4}\int_{B(R)}\int_{B(R)}\mathbb{E}[F_{\mu_A}(x)\cdot F_{\mu_A}(y)]\mathbb{E}[ F_{\mu_B}(x)F_{\mu_B}(y)]dxdy \nonumber \\
 =&  \frac{1}{\pi^2 R^4}\int_{B(R)}\int_{B(R)} r_{\mu_{B}}(x-y)r_{\mu_A}(x-y) dx dy\label{7}
 \end{align}
 where $r_{\mu_{B}}$ and $r_{\mu_A}$ are the covariance functions of  $F_{\mu_{B}}$ and $F_{\mu_A}$ respectively. Writing 
 \begin{align}
 r_{\mu}(x-y)= \int_{\mathbb{S}^1} e(\langle x-y, t\rangle)d\mu(t) \nonumber
 \end{align}
 and changing the order of integration, upon bearing in mind that $\mu_{A}$ is invariant under   rotation by $\pi$, shows that (\ref{7}) is equal to 
 
 \begin{align}
 &\frac{1}{\pi^2 R^4}\int_{B(R)}\int_{B(R)}\int_{\mathbb{S}^1}\int_{\mathbb{S}^1} e(\langle x-y, t- w\rangle)dxdy d\mu_{B}(t)d\mu_{A}(w) \nonumber \\
 =&  \int_{\mathbb{S}^1}\int_{\mathbb{S}^1} \left(\frac{J_1(R||t-w||)}{R||t-w||}\right)^2d\mu_{B}(t)d\mu_{A}(w) \label{23} 
 \end{align}
 where the second line follows from a similar computation to (\ref{8}). Now, we can split the double integral on the right hand side of (\ref{23}) as 
 \begin{align}
 \left(	\int \int_{||t-w||>R^{-1/2}} + 	\int \int_{||t-w||\leq R^{-1/2}}\right)\left(\frac{J_1(R||t-w||)}{R||t-w||}\right)^2d\mu_{B}(t)d\mu_{A}(w). \label{24}
 \end{align}
 As $J_1(T) \ll T^{-1/2}$ for $T$ large enough, we can bound the first integral in (\ref{24}) by $R^{-3/2}$. For the second integral in (\ref{24}) we observe that $J_1(T)/T=O(1)$ for $T$ small, so we fix $w$ to see that 
 \begin{align}
 \int_{||t-w||\leq R^{-1/2}}d\mu_{B}(t)=o(1) \label{35}
 \end{align}
  as $\mu_{B}$ has no atoms. Therefore, combining (\ref{34}), (\ref{24}) and (\ref{35}), we have shown that the third term in (\ref{16}) has zero mean and variance tending to $0$ as $R\rightarrow \infty$; thus, it converges in distribution to $0$. This proves (\ref{33}) and the Proposition follows. 
\end{proof}
\subsection{Concluding the proof of Theorem \ref{main corollary}} 
To deduce convergence of the moments of a random variable from its convergence in distribution, by Lebesgue dominated convergence Theorem, it is enough to show that the third moment is bounded. In our case, we prove the following: 
\begin{lem}
	\label{third moment}
Let $E$ be an integer representable as the sum of two squares such that the number of 6-tuples $(\xi_1,...,\xi_6)\in \mathcal{E}^6$ satisfying 

 \begin{align}
\xi_1+ \xi_2+ \xi_3+\xi_4+ \xi_5+\xi_6=0 \nonumber
\end{align}
is
\begin{align}
15 N^3 + O(N^{\gamma 6}) \nonumber
\end{align}
for some $0<\gamma<1/2$; moreover, suppose that $f$ as in \eqref{function} satisfies $A_2$. Then, for all fixed $R>1$ and $r= R/\sqrt{E}$, we have the uniform bound 

\begin{align}
\int_{\mathbb{T}^2}dx \left(\frac{1}{\pi r^2}\int_{B(x,r)}|f(y)|^2dy\right)^3=O(1). \nonumber
\end{align}

\end{lem}
\begin{proof}
Transforming the variables and expanding the integral, we can write 
\begin{align}
	\left(\frac{1}{\pi r^2}\int_{B(x,r)}|f(y)|^2dy\right)^3 &= \left(\frac{1}{\pi}\int_{B(1)} |f(x+ry)|^2dy\right)^3 \nonumber \\
	&=\left( \sum_{\xi,\xi'}a_{\xi}\overline{a_{\xi'}}e(\langle\xi- \xi',x\rangle) \frac{1}{\pi} \int_{B(1)}e(\langle\xi- \xi',ry\rangle)\right)^3 \nonumber \\
	&=\sum_{\substack{\xi_1, \xi_2, \xi_3 \\ \xi_1', \xi_2', \xi_3 '}}a_{\xi_1}\overline{a_{\xi_1'}}a_{\xi_2}\overline{a_{\xi_2'}}a_{\xi_3}\overline{a_{\xi_3'}} e(\langle \xi_1+ \xi_2+ \xi_3-\xi_1'- \xi_2'-\xi_3',x\rangle) \nonumber \\
	&\times \left(\frac{1}{\pi}\int_{B(1)}e(\langle\xi- \xi',ry\rangle)\right)^3. \nonumber
\end{align}
Since $\int_{\mathbb{T}^2}dxe(\langle \xi_1+ \xi_2+ \xi_3-\xi_1'- \xi_2'-\xi_3',x\rangle)= 1$ if $ \xi_1+ \xi_2+ \xi_3-\xi_1'- \xi_2'-\xi_3'=0$ and vanishes otherwise, we have 

\begin{align}
&\sum_{\substack{\xi_1, \xi_2, \xi_3 \\ \xi_1', \xi_2', \xi_3 '}}	\int_{\mathbb{T}^2}dx e(\langle \xi_1+ \xi_2+ \xi_3-\xi_1'- \xi_2'-\xi_3',x\rangle) \nonumber \\
&= |\{\xi_i \in \mathcal{E}(E) : \xi_1+ \xi_2+ \xi_3+\xi_4+ \xi_5+\xi_6=0\}| \nonumber
\end{align}
We call a solution of $\xi_1+ \xi_2+ \xi_3+\xi_4+ \xi_5+\xi_6=0$ diagonal  if it is given by pair-wise cancellation ( like $\xi_i= \xi_i'$ for $i=1,2,3$), we call all the other solutions off-diagonal and split the sum accordingly
\begin{align}
\int_{\mathbb{T}^2}dx \left(\frac{1}{r^2}\int_{B(x,r)}|f(y)|^2dy\right)^3 	= \sum_{\substack{ \xi_1+ \xi_2+ \xi_3-\xi_1'- \xi_2'-\xi_3'=0 \\ \text{diagonal}}}a_{\xi_1}\overline{a_{\xi_1'}}a_{\xi_2}\overline{a_{\xi_2'}}a_{\xi_3}\overline{a_{\xi_3'}} \nonumber \\
 \times\left(\int_{B(1)}e(\langle\xi- \xi',ry\rangle)dy\right)^3
+O\left(\sum_{\substack{ \xi_1+ \xi_2+ \xi_3-\xi_1'- \xi_2'-\xi_3'=0 \\ \text{off-diagonal}}}|a_{\xi_1}\overline{a_{\xi_1'}}a_{\xi_2}\overline{a_{\xi_2'}}a_{\xi_3}\overline{a_{\xi_3'}}|\right) \label{81}
\end{align}
where we have bounded $| \int_{B(1)}dy e(\langle\xi- \xi',ry\rangle)|\leq \pi$. The number of diagonal solutions is $6!/(3!\cdot 2^3)= 15$, so the main term in \eqref{81} is bounded by 
\begin{align}
	15 \pi^3 \sum_{\xi_1, \xi_2, \xi_3}|a_{\xi_1}|^2|a_{\xi_2}|^2|a_{\xi_3}|^2= 15 \left(\sum_{ \xi}|a_{\xi}|^2\right)^3= 15\pi^3 .\label{30}
\end{align}
  Thanks to the assumptions, we have  $|a_{\xi}|\leq \sqrt{u(N)/N}$ for all $\xi \in \mathcal{E}$ and  the number of off-diagonal solutions is bounded by $N^{\gamma 6}$. Thus, 

\begin{align}
	\sum_{\substack{ \xi_1+ \xi_2+ \xi_3-\xi_1'- \xi_2'-\xi_3'=0 \\ \text{off-diagonal}}}|a_{\xi_1}\overline{a_{\xi_1'}}a_{\xi_2}\overline{a_{\xi_2'}}a_{\xi_3}\overline{a_{\xi_3'}}| \leq  (u(N)/N)^3 \cdot N^{\gamma 6}= O(1). \label{31}
\end{align}
 Combining \eqref{81}, \eqref{30} and \eqref{31}, we conclude the lemma. 
\end{proof}

We are finally ready to conclude the proof of Theorem \ref{main corollary}:

\begin{proof}[Proof of Theorem \ref{main corollary}]
	 The first claim of the Theorem  follows from combining Proposition \ref{conv dis} and Proposition \ref{massdistribution}.  By Lemma \ref{third moment}, we deduce that the first two moments are uniformly integrable, therefore 
	\begin{align}
	\int_{\mathbb{T}^2}(M_f(x,r)-1)^2\rightarrow Var\left( \alpha + \beta W(\mu_{A})  \right)= \beta^2Var(W(\mu_A)) \nonumber
	\end{align}
which concludes the proof. 
\end{proof}

\section{Proof of Theorem \ref{main theorem}}
As mentioned in the introduction, in \cite{BU,BW} it was proved that $f(x)$, when considered in a small  neighbourhood of $x\in \mathbb{T}^2$ and averaged over $x$, approximates a Gaussian field. Formally, we fix some large parameter $R>1$ and write $f$ around the point $x$ as 

\begin{align}
	F_x(y)= f\left( x+ \frac{R}{\sqrt{E}}y\right) \nonumber
\end{align}
for $y\in [1/2,1/2]^2$. Then we have the following proposition: 

\begin{prop}
	\label{main prop}
Let $f$ be as in (\ref{function}), suppose that $E$ satisfies assumption $A1$, $N\rightarrow \infty$, and let $\epsilon_1, \epsilon_2>0$  be given. Then there exists some $B_0=B_0(\epsilon_1,\epsilon_2,R )$ and some $E_0= E_0(\epsilon_1,\epsilon_2,R ,B_0)$ such that for all $E>E_0$ and for all $B>B_0$ the following holds: 

\begin{enumerate}
	\item There exists a subspace $\Omega'\subset \Omega$ with $\mathbb{P}[\Omega']>1-\epsilon_2$  and  a measure-preserving function $\tau: \Omega'\rightarrow \mathbb{T}^2$ such that $vol(\mathbb{T}^2\backslash \tau(\Omega'))\leq \epsilon_2$. 
	\item Uniformly for all $f$ satisfying assumption $A2$,  
	\begin{align}
		\underset{y\in[-1/2,1/2]^2}{\sup}|F_{\tau(\omega)}(y)- F^R_{\mu_{f}}(y,w)|\leq \epsilon_1 \nonumber
	\end{align} 
for all $\omega\in \Omega'$, where the covariance of  $F^R_{\mu_{f}}(y)$ is given by 
	\begin{align}
		\mathbb{E}[F^R_{\mu_{f}}(x)F^R_{\mu_{f}}(y)]= \int_{\mathbb{S}^1} e(\langle\lambda, R(x-y)\rangle)d\mu_{f}(\lambda). \nonumber
	\end{align}
	
\end{enumerate}
\end{prop}

Assuming Proposition \ref{main prop}, which we will prove after, we prove the main theorem. 
\begin{proof}[Proof of Theorem \ref{main theorem}]
Let $\epsilon>0$, $t>0$ be fixed and let $X(\mu_{f})= \int_{B(R)}|F_{\mu_{f}}|^2/\pi R^2$. By L\'{e}vy continuity theorem, the statement of the theorem is equivalent to the inequality
	
	\begin{align}
		\left|\int_{\mathbb{T}^2} \exp(it M_f(x,r))dx- \mathbb{E}[\exp(itX(\mu_{f}))]\right|\leq \epsilon \label{12}
	\end{align}
	for all $E$ sufficiently large (depending on $\epsilon,t$ and $R$). First, we apply Proposition \ref{main prop} with $\epsilon_2=\epsilon/8$ and $\epsilon_1=\epsilon_1(\epsilon,t)$ to be chosen later, to get $\tau:\Omega'\rightarrow \mathbb{T}^2$ and $F^R_{\mu_{f}}$, for all $E>E_0$ (and $B>B_0$). Next, we apply Lemma \ref{not too big} with $\delta_1=\epsilon/8$ to obtain some $M=M(\epsilon)$ such that 
	
	\begin{align}
	\mathbb{P}\left(|	F^R_{\mu_{f}}|^2>M\right) \leq \epsilon/8 \nonumber. 
	\end{align}
Choosing $\epsilon_1= \sqrt{\epsilon}/(8Mt)$, we deduce that, outside an event of size at most $\epsilon/4$ and a subset of $\mathbb{T}^2$ of size at most $\epsilon/8$, we have 
	
		\begin{align}
	&\underset{y\in[-1/2,1/2]^2}{\sup}|F_{\tau(\omega)}(y)- F^R_{\mu_{f}}(y)|\leq \epsilon_1 &|	F^R_{\mu_{f}}|^2\leq M. \label{9}
	\end{align}
Since, by transformation  of variables,
	\begin{align}
	X(\mu_{f})= \frac{1}{\pi}\int_{B(1)} |F^R_{\mu_{f}}|^2dx\label{10}
	\end{align}
	arguing as in Lemma \ref{conv dis}, we see that conditions  (\ref{9}) imply
	\begin{align}
		\left|\exp(it M_f(\tau(\omega),r))- \exp\left(\frac{it}{\pi} \int_{B(1)} |F^R_{\mu_{f}}|^2dx  \right)\right|\leq \epsilon/2. \label{11}
	\end{align}
	Hence, combining (\ref{10}), (\ref{11}) and the fact that (\ref{9}) holds outside a set of size at most $\epsilon/4$ and volume $\epsilon/4$, we obtain (\ref{12}) and this concludes the proof. 
\end{proof}
\section{ Proof of Proposition \ref{main prop}: Bourgain's de-randomisation.}
\label{part 1}
The material of this section is contained, in various forms, in \cite{BU,BW}. We present it here for the convenience of the reader and since Proposition \ref{main prop}, as stated, does not appear in the literature. 
\subsection{Approximating $f$ is small squares}
The aim of this section is to approximate the function $f$ (in squares of side $R/\sqrt{E}$) by a more tractable function so we fix some large parameter $R>1$ and recall the notation 

\begin{align}
	F_x(y)= f\left( x+ \frac{R}{\sqrt{E}}y\right)= \sum_{ |\xi|^2=E} a_{\xi}e(\langle\xi,x\rangle)e\left(\left\langle\frac{\xi}{\sqrt{E}},Ry\right\rangle\right). \nonumber
\end{align}
 The points $\{\xi/\sqrt{E}\}$ lie on the unit circle, so as $N\rightarrow \infty$ they will accumulate, we first want to approximate these accumulation points. To this end, we pick some large parameter $K$ and divide the circle into arcs of length $1/2K$
\begin{align}
I_k=\left(\frac{k-1}{2K},\frac{k}{2K}\right] \nonumber
\end{align}
for $-K+1\leq k\leq K$ and let 
\begin{align}
\mathcal{E}^{(k)}=\{\xi\in \mathcal{E}: \xi/\sqrt{E}\in I_k\}.  \nonumber
\end{align}
 Now, we pick some small parameter $0<\delta<1$ and further subdivide  $\{\mathcal{E}^{(k)}\}$ accordingly to the measure $\mu_{f}$ as
 
 \begin{align}
 	&\mathcal{K}=\{-1+K\leq k\leq K: \mu_{f}(I_k)\geq \delta\} & \mathcal{G}=\cup_{k\not \in \mathcal{K}}\mathcal{E}^{(k)}. \nonumber
 \end{align}
 Finally, we approximate the points in $\mathcal{E}^{(k)}$ via the middle point $\zeta^k$ of $I_k$. Thus, we have obtained the functions 
 
 \begin{align}
 	&\tilde{F}_x(y)= \sum_{k\in \mathcal{K}}\sum_{\xi\in \mathcal{E}^{(k)}} a_{\xi}e(\langle\xi,x\rangle)e\left(\left\langle\frac{\xi}{\sqrt{E}}-\zeta^k,Ry\right\rangle\right)e(\zeta^k,Ry) \label{13} \\
 	 & \tilde{\psi}_x(y)=F_x(y)- \tilde{F}_x(y). \nonumber
 \end{align}
 Let us consider inner sum in (\ref{13}). Since $|\xi/\sqrt{E}-\zeta^k|$ can be made arbitrarily small by taking $K$ large, we approximate the whole sum by the term with $y=0$
\begin{align}
b_k(x)=\frac{1}{\mu_{f}(I_k)^{1/2}}\sum_{\xi\in \mathcal{E}^{(k)}}a_{\xi}e(\langle\xi,x\rangle) \label{bk}
\end{align}
obtaining the function
\begin{align}
\phi_x(y)=\sum_{k\in \mathcal{K}}\mu_{f}(I_k)^{1/2}b_k(x)e(\langle R\zeta^{k},y \rangle). \label{phi}
\end{align}
Here, we define $b_k$ and  $\phi_x$ so that $\int_{[-1/2,1/2]^2} |\phi_x(y)|^2 dy=1$ for all $x\in \mathbb{T}^2$. Moreover, to $\phi_x$ we associate the  ``spectral measure"
\begin{align}
	\mu_{K}= \sum_{k\in \mathcal{K}} \sum_{\xi\in \mathcal{E}^{(k)}} \mu_{f}(I_k) \delta_{\zeta^k}\Big/ \sum_{k\in \mathcal{K}} \sum_{\xi\in \mathcal{E}^{(k)}} \mu_{f}(I_k) .\label{muk}
\end{align}
The above approximations are justified by the following lemma \cite[Lemma 1]{BU} and \cite[Proposition 3.2]{BW}.
\begin{lem}
	\label{close det}
	Let $R>1$, $\epsilon_3,\epsilon_4>0$ be given. Then there exists $K_0=K_0(R,\epsilon_1,\epsilon_2)$, $\delta_0=\delta_0(R,K,\epsilon_3,\epsilon_4)$ and a subset $V \subset \mathbb{T}^2$ with $vol(V)\leq \epsilon_4$ such that for all $K>K_0$, $\delta<\delta_0$ and $x\in \mathbb{T}^2\backslash V$ we have 
	\begin{align}
	\underset{y\in[-1/2,1/2]^2}{\sup}	|F_x- \phi_x|\leq \epsilon_3. \nonumber
		\end{align}
\end{lem}

\subsection{Passage to random fields}
\label{passage to random fields}
The next step in the proof of Proposition \ref{main prop} is to show that the  $b_k$'s, in (\ref{bk}), simultaneously approximate $K$ i.i.d. complex Gaussian random variables. We briefly sketch the argument in order to highlight the importance of the assumptions $A1$ and $A2$. To show that the $b_k$'s approximate Gaussian random variables, via the central limit theorem, it is enough to prove that the pseudo-random variables $a_{\xi}e(\langle\xi,x\rangle)$ are asymptotically independent, when averaged over $x\in \mathbb{T}^2$, and satisfy some appropriate Lindeberg's condition. The  assumption A2 assures that the  Lindeberg's condition is satisfied. Thus, we are left to show  asymptotic independence. This can be done computing the moments of the pseudo-random vector $(e(\langle\xi,x\rangle))_{\xi \in \mathcal{E}}$ as
\begin{align}
\int_{\mathbb{T}^2}\left|\sum_{ \xi} e(\langle\xi,x\rangle)\right|^{2l} dx= \sum_{\xi_1,...\xi_{2k}}\int_{\mathbb{T}^2}e(\xi_1-\xi_2+...-\xi_{2l})dx.  \label{2}
\end{align}
By orthogonality, the right hand side of (\ref{2}) counts the number of $2l$-tubles $(\xi_1,...,\xi_{2l})$ such that  $\xi_1-\xi_2+...-\xi_{2k}=0$. Tuples of the form $\xi_1=\xi_2,...,\xi_{2k-1}=\xi_{2k}$ contribute to the integral, and their contribution is $2l!\cdot N^{l}/ {2^l l!}$. Moreover, by the assumption A1, all other contributions have lower order as $N\rightarrow \infty$. Thus, via the method of moments, the  pseudo-random variables $\{e(\langle\xi,x\rangle)\}_{\xi}$ are asymptotically independent. The central limit theorem then implies that the  $b_k$'s are simultaneously Gaussian. Therefore, $\phi_x$ in (\ref{phi}), when averaged over $x\in \mathbb{T}^2$,  approximates a Gaussian random field with spectral measure $\mu_{K}$ in \eqref{muk}. Quantifying the above discussion,  we have the following lemma  \cite[Lemma 2]{BU} and \cite[Proposition 3.3]{BW}.

\begin{lem}
	\label{close randm}
	Let $R,K>1$, $0<\delta<1$ be as above and $\epsilon_5, \epsilon_6>0$ be given, moreover  suppose that E satisfies assumption $A1$ and $N\rightarrow \infty$.  Then there exists  $B_0=B_0(\epsilon_5,\epsilon_6,R )$ and  $E_0= E_0(\epsilon_5,\epsilon_6,R ,B_0)$ such that for all $E>E_0$ and for all $B>B_0$ the following holds: 
	\begin{enumerate}
		\item There exists a subspace $\Omega''\subset \Omega$ with $\mathbb{P}[\Omega'']>1-\epsilon_6$  and  a measure-preserving function $\tau: \Omega''\rightarrow \mathbb{T}^2$ such that $vol(\mathbb{T}^2\backslash \tau(\Omega''))\leq \epsilon_6$. 
		\item Uniformly for all $f$ satisfying assumption $A2$, 
		\begin{align}
		\underset{y\in[-1/2,1/2]^2}{\sup}|\phi_{\tau(\omega)}(y)- F^R_{\mu_{K}}(y,\omega)|\leq \epsilon_5 \nonumber
		\end{align} 
		for all $\omega\in \Omega''$, where the covariance of  $F^R_{\mu_{K}}(y)$ is given by 
		\begin{align}
		\mathbb{E}[F^R_{\mu_{K}}(y)F^R_{\mu_{K}}(x)]= \int_{\mathbb{S}^1} e(\langle\lambda, R(x-y)\rangle)d\mu_{K}(\lambda). \nonumber
		\end{align}
	\end{enumerate}
\end{lem}

To conclude the proof of Proposition \ref{main prop}, it is enough to show that $F^R_{\mu_{K}}$ is close to $F^R_{\mu_{f}}$. This is the content of the following lemma: 

\begin{lem}
	\label{spectral con}
	Let $\mu_{K}$ be given as in (\ref{muk}). Then, 	as $K\rightarrow \infty$, we have 
	\begin{align}
		\mu_{K} \Rightarrow \mu_{f}. \nonumber
	\end{align}
\end{lem}
\begin{proof}
	Let $A\subset \mathbb{S}^1$ be an open subset then  
	\begin{align}
	\mu_{f}(A)=\sum_{\xi \in A}|a_{\xi}|^2. \nonumber
	\end{align}
	On the other hand
	\begin{align}
	\mu_K(A)&=\sum_{\substack{\zeta^{(k)}\in A \\ k\in \mathcal{K}}}\mu_{f}(I_k)\Big/\sum_{k\in \mathcal{K}}\mu_{f}(I_k)
	= \sum_{\substack{\zeta^{(k)}\in A \\ k\in \mathcal{K}}} \sum_{\xi \in I_k}|a_{\xi}|^2\Big/\sum_{k\in \mathcal{K}} \mu_{f}(I_k) .\nonumber
	\end{align}
	By definition of $\mathcal{K}$, we have
	\begin{align}
	&\sum_{k\in \mathcal{K}} \mu_{f}(I_k)=1 -\sum_{k\not\in \mathcal{K}}\mu_{f}(I_k)=1+O(\delta K) \nonumber 
	\end{align}
	and similarly 
	\begin{align}
	\sum_{\substack{\zeta^{(k)}\in A \\ k\in \mathcal{K}}} \sum_{\xi \in I_k}|a_{\xi}|^2= \sum_{\xi \in A}|a_{\xi}|^2 +O(\delta K). \nonumber
	\end{align} 
	Hence, taking $\delta<K^{-2}$,   we obtain
	\begin{align}
	\mu_K(A)=\mu_{f}(A) + O\left(\frac{1}{K}\right) \nonumber
	\end{align}  
	as required. 
\end{proof}
We are now ready to complete the proof of Proposition \ref{main prop}: 
\begin{proof}[Proof of Proposition \ref{main prop}]
	Let $\epsilon_1,\epsilon_2$ be given, then we apply Lemma \ref{Sodin} together with Lemma \ref{spectral con}  with $\alpha=\epsilon_1/4$ and $\delta_2= \epsilon_2/4$ to see that 
	\begin{align}
			\underset{y\in[-1/2,1/2]^2}{\sup}|F^R_{\mu_{f}}(y)- F^R_{\mu_{K}}(y)|\leq \epsilon_1/4 \label{ 13}
	\end{align}
	for all $K>K_0$ outside an event $\Omega''''$ of probability at most $ \epsilon_2/8$. Now, apply Lemma \ref{close randm} with  $\epsilon_5= \epsilon_1/4$ and $\epsilon_6= \epsilon_2/4$ to get  $\Omega'''$ and $\tau$,  taking $E>E_0$ and $B>B_0$. Define $\Omega''= \Omega''' \backslash \Omega''''$, since $\mathbb{P}[\Omega'''']\leq \epsilon_2/4$ and $\mathbb{P}[\Omega''']>1-\epsilon_2/4$,  $\mathbb{P}[\Omega'']>1-\epsilon_1/2$. So, by Lemma  \ref{close det} and the triangular inequality, we have
	
	\begin{align}
\underset{y\in[-1/2,1/2]^2}{\sup}|\phi_{\tau(\omega)}(y)- F^R_{\mu_{f}}(y,\omega)|\leq \epsilon_1/2 \label{14}
	\end{align} 
	for all $\omega\in \Omega''$. Now let $\epsilon_3=\epsilon_1/2$ and $\epsilon_4=\epsilon_2/4$ in Lemma \ref{close det} to obtain a set $V$ of volume at most $\epsilon_2/4$ such that 
		\begin{align}
	|F_x- \phi_x|\leq \epsilon_1/2 \label{15}
	\end{align}
	for all $x\in \mathbb{T}^2\backslash V$. Since $\tau$ is measure-preserving, we have $\mathbb{P}[\tau^{-1}(V)]\leq \epsilon_2/2$ so we finally take $\Omega'
	= \Omega'' \backslash \tau^{-1}(V)$, note that $\mathbb{P}[\Omega']>1-\epsilon_2$. For all $\omega \in \Omega'$, in light of (\ref{14}) and (\ref{15}), we have 
	\begin{align}
		\underset{y\in[-1/2,1/2]^2}{\sup}|F_{\tau(\omega)}(y)- F^R_{\mu_{f}}(y,\omega)|\leq \epsilon_1 \nonumber
	\end{align}
	as required. 
\end{proof}

\section*{Acknowledgement}
The author would like to thank Igor Wigman for pointing out the question considered here and for the many discussions as well as Nadav Yesha for useful conversations about his work. This work was supported by the Engineering and Physical Sciences Research Council [EP/L015234/1].  
The EPSRC Centre for Doctoral Training in Geometry and Number Theory (The London School of Geometry and Number Theory), University College London.

\end{document}